\documentclass[12pt]{amsproc}

\usepackage{amsmath, amssymb}

\newtheorem{theorem}{\sc Theorem}[section]

\newtheorem{lemma}[theorem]{\sc Lemma}

\newtheorem{corollary}[theorem]{\sc Corollary}

\usepackage{color}

\begin{document}
\title[On finite-by-nilpotent groups]
{On finite-by-nilpotent groups}

\author[E. Detomi]{Eloisa Detomi}

\address{Dipartimento di Ingegneria dell'Informazione - DEI, Universit\`a di Padova, Via G. Gradenigo 6/B, 35121 Padova, Italy} 
\email{eloisa.detomi@unipd.it}

\author[G. Donadze]{Guram Donadze}
\address{ Department of Mathematics, University of Brasilia, Brasilia-DF, 70910-900 Brazil
and
Institute of Cybernetics of the Georgian Technical University, Sandro Euli Str. 5, 0186, Tbilisi, Georgia }
\email{gdonad@gmail.com}

\author[M. Morigi]{Marta Morigi}
\address{Dipartimento di Matematica, Universit\`a di Bologna, 
Piazza di Porta San Donato 5, 40126 Bologna, Italy}
\email{marta.morigi@unibo.it}

\author[P. Shumyatsky]{Pavel Shumyatsky}
\address{Department of Mathematics, University of Brasilia,
Brasilia-DF,   70910-900 Brazil }
\email{pavel2040@gmail.com}

\thanks{The first and third authors are members of INDAM. The fourth author was supported by CNPq-Brazil}

\begin{abstract}
\noindent  Let $\gamma_n=[x_1,\dots,x_n]$ be the $n$th lower central word. Denote by $X_n$ the set of $\gamma_n$-values in a group $G$ and suppose that there is a number $m$ such that $|g^{X_n}|\leq m$ for each $g\in G$. We prove that $\gamma_{n+1}(G)$ has finite $(m,n)$-bounded order. This generalizes the much celebrated theorem of B. H. Neumann that says that the commutator subgroup of a BFC-group is finite.
\end{abstract}
\maketitle

\section{Introduction}

Given a group $G$ and an element $x\in G$, we write $x^G$ for the conjugacy class containing $x$. Of course, if the number of elements in $x^G$ is finite, we have $|x^G|=[G:C_G(x)]$. A group is said to be a BFC-group if its conjugacy classes are finite and of bounded size. One of the most famous of B. H. Neumann's theorems says that in a BFC-group the commutator subgroup $G'$ is finite \cite{bhn}. It follows that if $|x^G|\leq m$ for each $x\in G$, then the order of $G'$ is bounded by a number depending only on $m$. A first explicit bound for the order of $G'$ was found by J. Wiegold \cite{wie}, and the best known was obtained in 
\cite{gumaroti} (see also \cite{neuvoe} and \cite{sesha}).

The recent articles \cite{dieshu} and \cite{dms17} deal with groups $G$ in which conjugacy classes containing commutators are bounded. Recall that  multilinear commutator words are words which are obtained by nesting commutators, but using always different variables. More formally, the group-word $w(x)=x$ in one variable is a multilinear commutator; if $u$ and $v$ are  multilinear commutators involving different variables then the word $w=[u,v]$ is a multilinear commutator, and all multilinear commutators are obtained in this way. Examples of multilinear commutators include the familiar lower central words $\gamma_n(x_1,\dots,x_n)=[x_1,\dots,x_n]$ and derived words $\delta_n$, on $2^n$ variables, defined recursively by
$$\delta_0=x_1,\qquad \delta_n=[\delta_{n-1}(x_1,\ldots,x_{2^{n-1}}),\delta_{n-1}(x_{2^{n-1}+1},\ldots,x_{2^n})].$$ We let $w(G)$ denote the verbal subgroup of $G$ generated by all $w$-values. 
Of course, $\gamma_n(G)$ is the $n$th term of the lower central series of $G$ while $\delta_n(G)=G^{(n)}$ is the $n$th  term of the derived series.

The following theorem was established in \cite{dms17}.

\begin{theorem}\label{17}
Let $m$ be a positive integer and $w$ a multilinear commutator word. Suppose that $G$ is a group in which $|x^G|\leq m$ for any $w$-value $x$. Then the order of the commutator subgroup of $w(G)$ is finite and $m$-bounded. 
\end{theorem}
Throughout the article we use the expression ``$(a,b,\dots)$-bounded" to mean that a quantity is finite and bounded by a certain number depending only on the parameters $a,b,\dots$.

The present article grew out of the observation that a modification of the techniques developed in \cite{dieshu} and \cite{dms17} can be used to deduce that if $|x^{G'}|\leq m$ for each $x\in G$, then $\gamma_3(G)$ has finite $m$-bounded order. Naturally, one expects that a similar phenomenon holds for other terms of the lower central series of $G$. This is indeed the case.  

\begin{theorem}\label{main} Let $m,n$ be positive integers and $G$ a group. If $|x^{\gamma_n(G)}|\le m$ for any $x \in G$, then $\gamma_{n+1}(G)$ has finite $(m,n)$-bounded order. 
\end{theorem}

Using the concept of verbal conjugacy classes, introduced in \cite{dfs}, one can obtain a generalization of Theorem \ref{main}. Let $X_n=X_n(G)$ denote the set of $\gamma_n$-values in a group $G$. It was shown in \cite{brakrashu} that if $|x^{X_n}|\le m$ for each $x\in G$, then $|x^{\gamma_n(G)}|$ is $(m,n)$-bounded. Hence, we have

\begin{corollary}\label{gamma} Let $m,n$ be positive integers and $G$ a group. If $|x^{X_n(G)}|\le m$ for any $x \in G$, then $\gamma_{n+1}(G)$ has finite $(m,n)$-bounded order. 
\end{corollary}
Observe that Neumann's theorem can be obtained from Corollary \ref{gamma} by specializing $n=1$. Another result which is straightforward from Corollary \ref{gamma} is the following characterization of finite-by-nilpotent groups.
\begin{theorem}\label{char} A group $G$ is finite-by-nilpotent if and only if there are positive integers $m,n$ such that $|x^{X_n}|\le m$ for any $x\in G$.
\end{theorem}

\section{Preliminary results}
Recall that in any group $G$ the following ``standard commutator identities" hold, when $x,y,z\in G$.

\begin{enumerate}
\item $[xy,z]=[x,z]^y[y,z]$ 
\item $[x,yz]=[x,z][x,y]^z$ 
\item $[x,y^{-1},z]^y[y,z^{-1},x]^z [z,x^{-1},y]^x=1$  (Hall-Witt identity); 
\item $[x,y,z^x][z,x,y^z][y,z,x^y]=1.$
\end{enumerate}
Note that the fourth identity follows from the third one. Indeed, we have
$$[x^y,y^{-1},z^y] [y^z,z^{-1},x^z][z^x,x^{-1},y^x]=1.$$ 
Since $[x^y,y^{-1}]= [y,x]$, it follows that
$$[y,x,z^y] [z,y,x^z][x,z,y^x]=1.$$ 

Recall that $X_i$ denote the set of $\gamma_i$-values in a group $G$.

\begin{lemma}\label{conj2} 
Let $k,n$ be integers with $2\le k\le n$ be integers and let $G$ be a group such that $[\gamma_k(G), \gamma_{n}(G)]$ is finite and $|x^{\gamma_n(G)}|\le m$ for any $x\in G$. Then for every $g\in X_n$ we have
 \[ |g^{\gamma_{k-1}(G)}|\le  m^{n-k+2}|[\gamma_k(G),\gamma_{n}(G)]|. \]
\end{lemma}

\begin{proof} Let $N=[\gamma_k(G), \gamma_{n}(G)]$. It is sufficient to prove that in the quotient group $G/N$, for every integer $d$ with  $k-1 \le d\le n$
 \[ |(gN)^{\gamma_{d}(G/N)}| \le m^{n-d+1} \quad  \textrm{for every
  $\gamma_{n-d+1}$-value $gN \in G/N$},\] 
since this implies that $g^{\gamma_{d}(G)}$ is contained at most 
 $m^{n-d+1}$ cosets of $N$, whenever $g \in X_{{n-d+1}}$.
  
So in what follows we assume that $N=1$. The proof is by induction on $n-d$. The case $d=n$ is immediate from the hypotheses. 
 
Let $c=n-d+1$. Choose $g\in X_c$ and write $g=[x,y]$ with $x\in X_{c-1}$ and $y\in G$. Let $z \in\gamma_d(G)$. We have \[ [x,y,z^x][z,x,y^z][y,z,x^y]=1. \]
 Note that 
 \[[z,x]\in [\gamma_{d}(G),\gamma_{c-1}(G)] \le \gamma_{d-1+c}(G)=\gamma_n(G)\]
and 
\[ [y,z] \in \gamma_{d+1}(G) \le \gamma_k(G),\]
whence $[z,x,y^z]=[z,x,y [y,z]] = [z,x,y].$ 
Thus,
\begin{eqnarray*}
 1= [x,y,z^x][z,x,y^z][y,z,x^y] 
  &=& [x,y]^{-1}[x,y]^{z^x}[z,x,y][y,z,x^y]  \\
  &=& [x,y]^{-1}[x,y]^{z^x} (y^{-1})^{[z,x]} y ((x^y)^{-1})^{[y,z]}x^y.   
\end{eqnarray*}
It follows that
 \[ [x,y]^{z^x} = [x,y](x^{-1})^y(x^y)^{[y,z]}y^{-1} y^{[z,x]}.\]
 Since $x^y \in X_{c-1}$ and $[y,z] \in \gamma_{d+1}(G)$, by induction  
 $$|\{(x^y)^{[y,z]} \mid z \in \gamma_d(G)\}| \le m^{n-d-1+1}.$$ Moreover, 
 $[z,x] \in \gamma_n(G)$ an so $|\{y^{[z,x]} \mid z\in \gamma_d(G)\}| \le m$. 
 Thus,
 \[ |\{[x,y]^{z^x}\mid  z \in \gamma_d(G)\}|=
  |\{[x,y]^{z}\mid  z \in \gamma_d(G)\} |\le m m^{n-d}=m^{n-d+1} 
 \]
as claimed. 
 \end{proof}

Let $H$ be a group generated by a set $X$ such that $X=X^{-1}$. Given
an element $g\in H$,
 we write $l_X(g)$ for the minimal number $l$ with the
property that $g$ can be written as a product of $l$ elements of $X$. Clearly, $l_X(g)=0$ if and only if $g=1$. We call $l_X(g)$ the length of $g$ with respect to $X$. The following result is Lemma 2.1 in \cite{dieshu}.

\begin{lemma}\label{2.1g}  Let $H$ be a group generated by a set $X=X^{-1}$ and let $K$ be a subgroup of finite index $m$ in $H$. Then each coset $Kb$ contains an element $g$ such that $l_X(g)\le m-1$.
\end{lemma}

In the sequel the above lemma will be used in the situation where $H=\gamma_n(G)$ and $X=X_n$ is the set of $\gamma_n$-values in $G$. Therefore we will write $l(g)$ to denote the smallest number such that the element 
 $g\in \gamma_n(G)$
 can be written as a product of as many $\gamma_n$-values.

Recall that if $G$ is a group, $a\in G$ and $H$ is a subgroup of $G$, then $[H,a]$ denotes the subgroup of $G$ generated by all commutators of the form $[h,a]$, where $h\in H$. It is well-known that $[H,a]$ is normalized by $a$ and $H$.
\begin{lemma} \label{2.3} 
Let $k,m,n\ge2$ and let $G$ be  a group in which $|x^{\gamma_n(G)}|\le m$ for any $x\in G$. Suppose that $[\gamma_k(G),\gamma_{n}(G)]$ is finite. 
 Then for every $x\in\gamma_{k-1}(G)$ the order of $[\gamma_{n}(G),x]$ is bounded in terms of $m$, $n$ and $|[\gamma_k(G),\gamma_{n}(G)]|$ only. 
\end{lemma}
\begin{proof} By Neumann's theorem $\gamma_n(G)'$ has $m$-bounded order, so the statement is true for $k\ge n+1$. 
Therefore we deal with the case $k\le n$. Without loss of generality we can assume that $[\gamma_k(G), \gamma_{n}(G)]=1$.

Let  $x \in \gamma_{k-1}(G)$. Since $|x^{\gamma_n(G)}|\le m$, the index of $C_{\gamma_n(G)}(x)$ in $\gamma_n(G)$ is at most $m$ 
 and  by Lemma \ref{2.1g} we can choose elements $y_1,\dots,y_m \in X_n$
 such that $l(y_i)\le m-1$ and $[\gamma_n(G),x]$ is generated by the commutators $[y_{i},x]$. For each
$i=1,\dots,m$ write $y_i=y_{i\,1}\cdots y_{i\,{m-1}}$, where $y_{i\,j}\in X_n$. The standard commutator identities show that $[y_i,x]$ can be written as a product of conjugates in $\gamma_n(G)$ of the commutators $[y_{ij} ,x]$. 
Since $[y_{ij} ,x] \in \gamma_{k}(G)$, for any $z \in \gamma_n(G)$ we have 
 that 
 \[[[y_{ij} ,x], z] \in [\gamma_k(G), \gamma_{n}(G)]=1.\]
  Therefore $[y_i,x]$ can be written as a product of the commutators $[y_{ij} ,x]$.
 
 Let $ T = \langle x,y_{ij} \mid 1\le i, j \le m \rangle$. It is clear that $[\gamma_n(G), x]\le T'$ and
so it is sufficient to show that $T'$ has finite $(m,n)$-bounded order. 
 Observe that $T \le \gamma_{k-1}(G)$. By Lemma \ref{conj2}, $C_{\gamma_{k-1}(G)}(y_{ij})$ has $(m,n)$-bounded index in $\gamma_{k-1}(G)$. It follows that $C_{T}(\{y_{ij}\mid1\le i,j\le m \})$ has $(m,n)$-bounded index in $T$. Moreover, $T\le\langle x\rangle\gamma_n(G)$ and $|x^{\gamma_n(G)}|\le m$, whence $|T:C_T(x)|\le m$.  Therefore the centre of $T$ has $(m,n)$-bounded index in $T$. Thus, Schur's theorem \cite[10.1.4]{robinson} tells us that $T'$ has finite $(m,n)$-bounded order, as required.
\end{proof}

The next lemma can be seen as a development related to Lemma 2.4 in \cite{dieshu} and Lemma 4.5 in \cite{wie}. It plays a central role in our arguments.  

\begin{lemma} \label{basic-light}  
Let $k,n\ge2$. Assume that $|x^{\gamma_n(G)}|\le m$ for any $x\in G$. Suppose that $[\gamma_k(G),\gamma_{n}(G)]$ is finite. Then the order of $[\gamma_{k-1}(G),\gamma_{n}(G)]$  is bounded in terms of $m$, $n$ and $|[\gamma_k(G),\gamma_{n}(G)]|$ only.
\end{lemma}
\begin{proof} Without loss of generality we can assume that $[\gamma_k(G), \gamma_{n}(G)]=1$.
Let $W=\gamma_n(G)$. Choose an element $a\in X_{k-1}$ such that the number of conjugates of $a$ in $W$ is maximal possible, that is,
$r=|a^W|\ge|g^W|$ for all $g\in X_{k-1}$.

 By Lemma \ref{2.1g}  we can choose $b_1,\dots, b_r\in W$ such that $l(b_i) \le m-1$ and $a^W = \{a^{b_i} | i = 1, \dots, r\}$.
%
%
 Let $K=\gamma_{k-1}(G)$. 
 Set $M=(C_K(\langle b_1 ,\dots, b_r\rangle ))_K$ (i.e. $M$ is the intersection of  all $K$-conjugates of $C_K(\langle b_1,\dots,b_r\rangle)$).
 Since $l(b_i) \le m-1$ and, by Lemma \ref{conj2}, 
  $C_K(x)$ has $(m,n)$-bounded index
    in $K$ for each  $x\in X_n$, 
 the subgroup  $C_K(\langle b_1, \dots , b_r\rangle)$ has $(m,n)$-bounded index in $K$,  so also $M$ has $(m,n)$-bounded index in $K$.

Let $v\in M$. Note that  $(va)^{b_i} = va^{b_i}$ for each $i =1,\dots, r$. Therefore the elements $va^{b_i}$ form the conjugacy class $(va)^W$ because they are all different and their number is the allowed maximum. So, for an arbitrary element $h\in W$ there exists $b\in\{b_1 ,\dots, b_r\}$ such that
$(va)^h= va^b$ and hence $v^h a^h = va^b$. Therefore $[h, v] = v^{-h}v=a^h a^{-b}$ and so $[h, v]^a =a^{-1} a^h a^{-b} a = [a,h][b,a] \in [W,a].$
Thus $[W,v]^a \le [W,a]$ and  so $[W,M]\le [W,a].$

Let $x_1, \dots , x_s$  be a set of coset representatives of $M$ in $K$. As  $[W,x_i]$ is normalized by $W$ for each $i$, it follows that
\[ [W, K] \le [W,x_1] \cdots [W,x_s] [W,M] \le  [W,x_1] \cdots [W,x_s] [W,a] .\] 
Since  $s$ is $(m,n)$-bounded and by Lemma \ref{2.3} the orders of all subgroups  $[W,x_i]$ and 
 $[W,a]$ are bounded in terms of $m$ and $n$  only, 
 the result follows.   
\end{proof}

\begin{proof}[Proof of Theorem \ref{main}.] Let $G$ be a group in which $|x^{\gamma_n(G)}|\le m$ for any $x\in G$. We need to show that $\gamma_{n+1}(G)$ has finite $(m,n)$-bounded order. 
We will show that the order of $[\gamma_k(G),\gamma_{n}(G)]$ is finite and $(m,n)$-bounded for $k=n,n-1,\dots,1$. 
This is sufficient for our purposes since $[\gamma_1(G),\gamma_{n}(G)]=\gamma_{n+1}(G)$. We argue by backward induction on $k$. The case $k=n$ is immediate from Neumann's theorem so we assume that $k\le n-1$ and the order of $[\gamma_{k+1}(G),\gamma_{n}(G)]$ 
is finite and $(m,n)$-bounded. 
 Lemma \ref{basic-light} now shows that also the order of $[\gamma_k(G),\gamma_{n}(G)]$ is finite and $(m,n)$-bounded, as required.
\end{proof}

\begin{proof}[Proof of Corollary \ref{gamma}.] Let $G$ be a group in which $|x^{X_n(G)}|\le m$ for any $x\in G$. We wish to show that $\gamma_{n+1}(G)$ has finite $(m,n)$-bounded order. Theorem 1.2 of \cite{brakrashu} tells us that $|x^{\gamma_n(G)}|$ is $(m,n)$-bounded. The result is now immediate from Theorem \ref{main}.
\end{proof}

\begin{proof}[Proof of Theorem \ref{char}.] In view of Corollary \ref{gamma} the theorem is self-evident since a group $G$ is finite-by-nilpotent if and only if some term of the lower central series of $G$ is finite.
\end{proof}

\end{document}